\documentclass[a4paper,12pt]{llncs}
\usepackage{amsmath,amssymb}
\usepackage{hyperref}
\usepackage{epsfig}

\usepackage{color}
\usepackage{graphicx}
\usepackage{subfigure}
\usepackage{wrapfig}
\usepackage{amssymb}
\usepackage{amsmath}
\usepackage{tikz, hyperref}

\newcommand{\vol}{{\rm vol}}

\date{}
\begin{document}
\title{\Large{Common adversaries form alliances: modelling complex networks via anti-transitivity}\thanks{Research supported by grants from NSERC and Ryerson University.}}
\author{Anthony Bonato \inst{1}\and Ewa Infeld \inst{1}\and \\ Hari Pokhrel \inst{1} \and Pawe\l{} Pra\l{}at  \inst{1}}
\institute{Ryerson University}
\maketitle
\begin{abstract}
Anti-transitivity captures the notion that enemies of enemies are friends, and arises naturally in the study of adversaries in social networks and in the
study of conflicting nation states or organizations. We present a simplified, evolutionary model for anti-transitivity influencing link formation in complex networks, and analyze the model's
network dynamics. The Iterated Local Anti-Transitivity (or ILAT) model creates anti-clone nodes in each time-step, and joins anti-clones to the parent node's non-neighbor set. The graphs generated by
ILAT exhibit familiar properties of complex networks such as densification, short distances (bounded by absolute constants), and bad spectral expansion. We determine the cop and domination number
for graphs generated by ILAT, and finish with an analysis of their clustering coefficients. We interpret these results within the context of real-world complex networks and present open
problems.
\end{abstract}

\section{Introduction}\label{intro}

Transitivity is a pervasive and folkloric notion in social networks, summarized in the adage that ``friends of friends are more likely friends''. A simplified, deterministic model for transitivity
was posed in \cite{ilt1,ilt}, where nodes are added over time, and each node's \emph{clone} is adjacent to it and all of its neighbors. The resulting Iterated Local Transitivity (or ILT) model, while
elementary to define, simulates many properties of social and other complex networks. For example, as shown in \cite{ilt}, graphs generated by the model densify over time, have the small world
property (that is, small distances and high local clustering), and exhibit bad spectral expansion. For further properties of the ILT model, see \cite{BJR,mason}

Complex networks contain numerous mechanisms governing link formation, however. Structural balance theory in social network analysis cites several mechanisms to complete triads \cite{sack}. Another
folkloric adage is that ``enemies of enemies are more likely friends''. Adversarial relationships may be modelled by non-adjacency, and so we have the resulting closure of the triad as described in
Figure~\ref{at}.

\begin{figure}[h!]
\begin{centering}
    \includegraphics{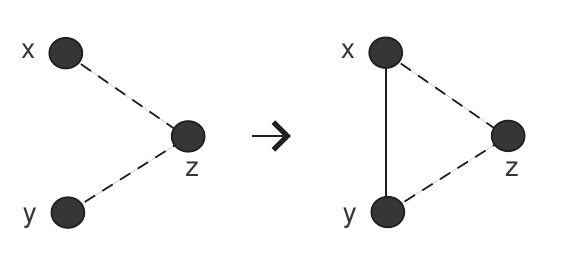}
    \caption{Nodes $x$ and $y$ share $z$ as a mutual adversary, and so form an alliance.}
    \label{at}
    \end{centering}
\end{figure}

Such triad closure is suggestive of an analysis of adversarial relationships between nodes as one mechanism for link formation. For instance, in social networks, we may consider both friendship ties and enmity (or rivalry) between actors. We may also consider
opposing networks of nation states or rival organizations, and consider alliances formed by mutually shared adversaries.  See \cite{guo} for a recent study using the spatial location of cities to form
an interaction network, where links enable the flow of cultural influence, and may be used to predict the rise of conflicts and violence. Another example comes from market graphs, where the nodes are stocks, and stocks are
adjacent as a function of their correlation measured by a threshold value $\theta \in (0,1).$ Market graphs were considered in the case of negatively correlated (or adversarial) stocks, where stocks are adjacent if $\theta < \alpha,$ for some positive $\alpha$; see \cite{market}.

In the present paper, we consider a simplified, deterministic model for anti-transitivity in complex networks. The Iterated Local Anti-Transitivity (or ILAT) model duplicates nodes in each time-step
by forming \emph{anti-clone} nodes, and joins them to the parent node's non-neighbor set. We give a precise definition of the model below in the next section. Perhaps unexpectedly, graphs generated
by ILAT model exhibit familiar properties of complex networks such as densification, small world properties, and bad spectral expansion (analogously to, but different from properties exhibited by
ILT).

We organize the discussion in this extended abstract as follows. In Section~\ref{model}, we give a precise definition of the ILAT model and examine its basic properties. We prove that graphs
generated by ILAT densify over time. We derive the density of ILAT graphs, and consider their degree distribution. In Section~\ref{distance}, we prove that ILAT graphs have diameter 3 for
sufficiently large time-steps (regardless of the initial graph). Further, we determine after several time-steps, ILAT graphs have cop number 2 and domination number 3. We include in
Section~\ref{cluster} an analysis of the clustering coefficients and provide upper and lower bounds. The final section interprets our results within real-world complex networks, and presents
open problems derived from the analysis of the model.

We consider undirected graphs throughout the paper. For background on graph theory, the reader is directed to~\cite{west}. Additional background on complex networks may be found in the book
\cite{bbook}.

\section{The ILAT model}\label{model}

The Iterated Local Anti-Transitivity (or ILAT) model generates a sequence $(G_t:t\ge 0)$ of graphs over a sequence of discrete time-steps. The one parameter of the model is the initial graph $G_0$.
Assuming the graph at time $G_t$ is defined, we define $G_{t+1}$ as follows. For a given node $x \in V(G_t)$, define its \emph{anti-clone} $x'$ as a new node adjacent to non-neighbors of $x.$ More
precisely, $x'$ is adjacent to all nodes in $N^c(x)$, where $N^c(x) = \{y \in V(G_t): xy \not \in E(G)\}.$ To form $G_{t+1}$, to each node $x$ add its anti-clone $x'.$

The intuition behind that model is that the anti-clone $x'$ is adversarial with $x$, and non-neighbors of $x$ (that is, its own adversaries) become allied with $x'.$ This process, therefore,
iteratively applies the triad closure in Figure~\ref{at}. Note that the number of nodes doubles in each time-step, and the set of anti-clones forms an independent set. See Figure~\ref{step} for an
example.

\begin{figure}[h!]
\begin{centering}
    \includegraphics[scale=0.2]{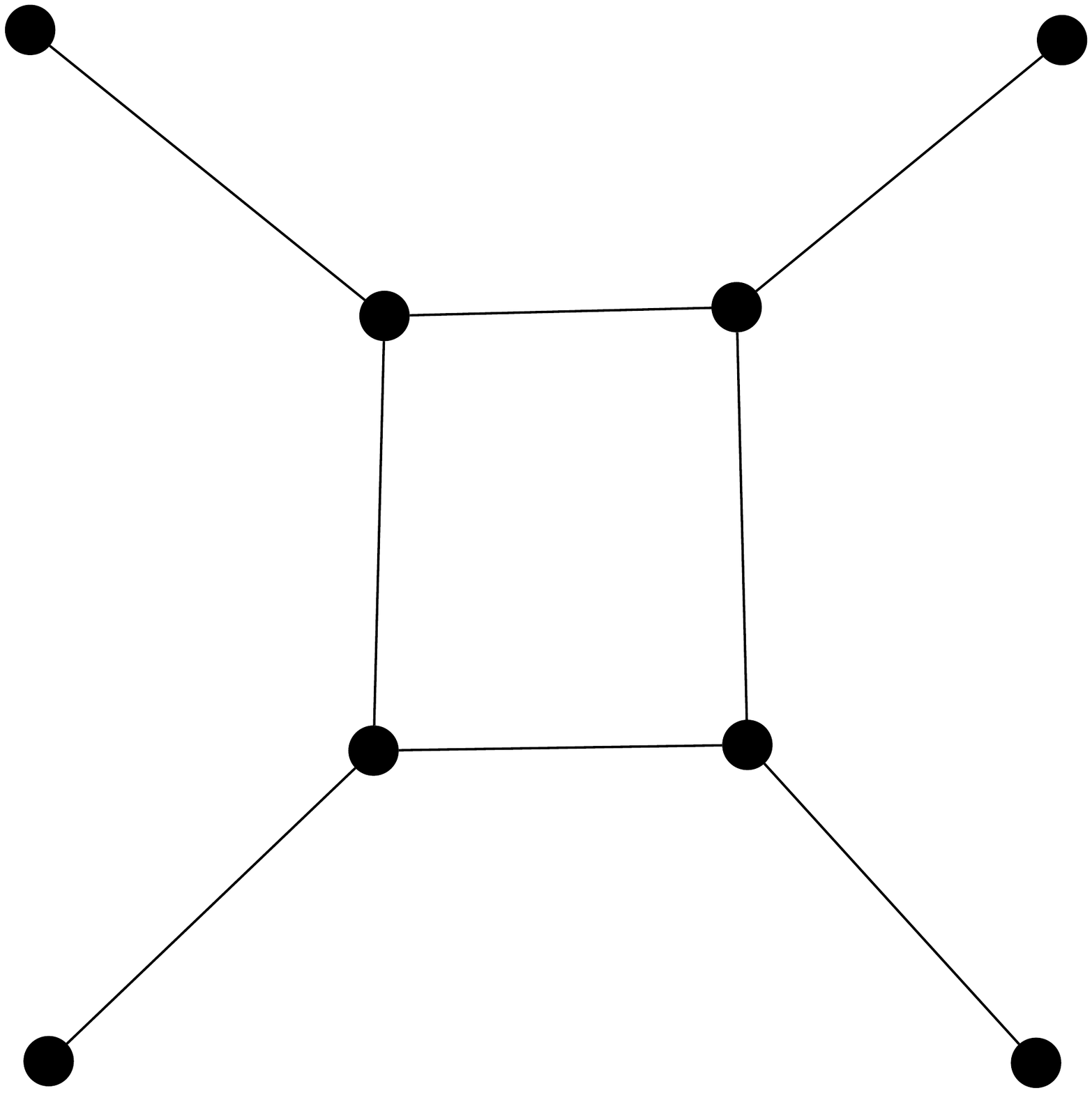}
    \includegraphics[scale=0.2]{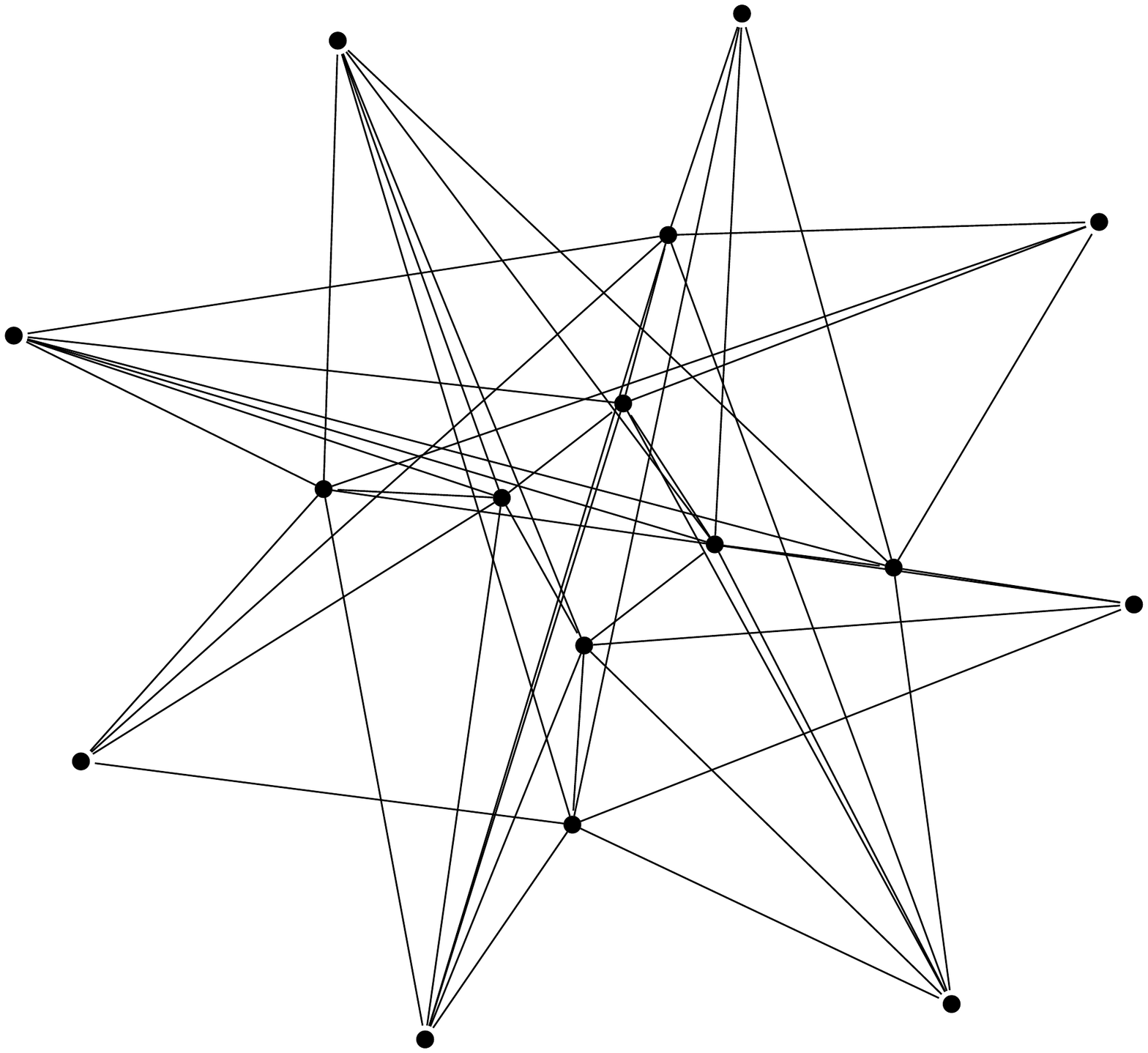} \\
    \includegraphics[scale=0.2]{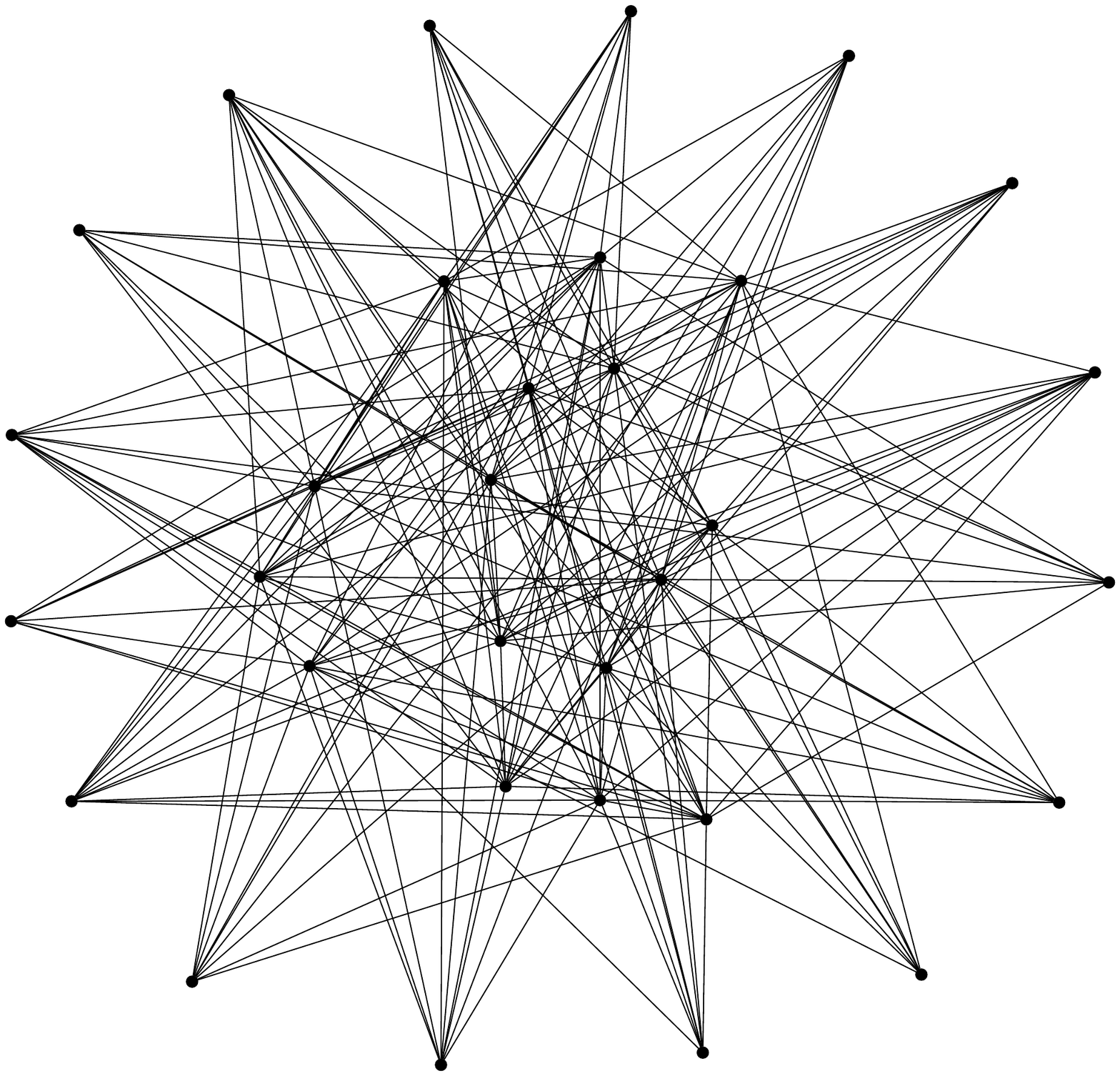}
    \includegraphics[scale=0.2]{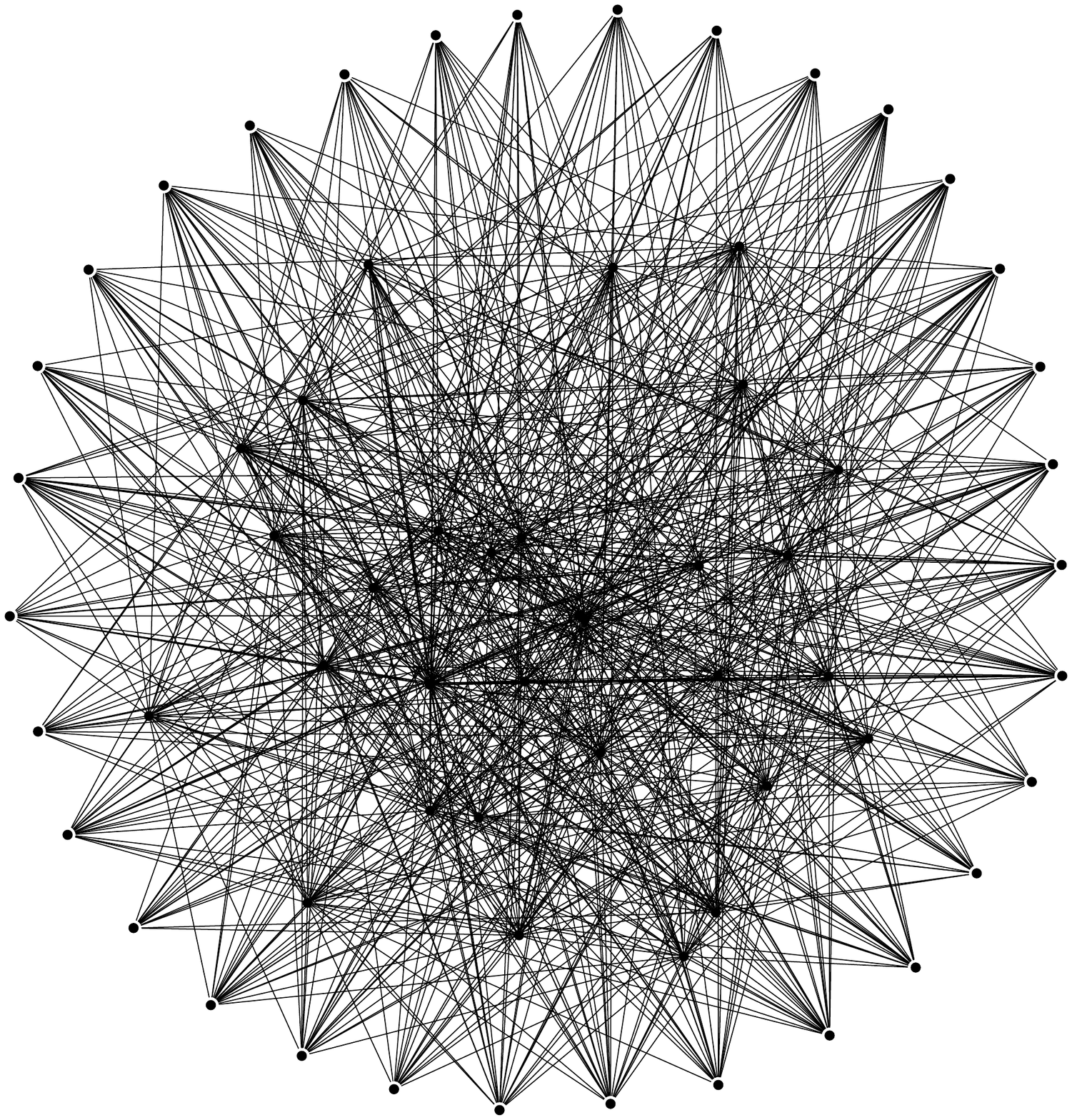}
    \caption{An example of the first four time-steps of the ILAT model, where the initial graph is the four-cycle $C_4.$}
\label{step}
    \end{centering}
\end{figure}

We introduce some simplifying notation. Let $n_t$ be the number of nodes at time $t$, $e_t$ be the number of edges at time $t,$ and the degree of a node $x$ at time $t$ will be denoted $\deg_t(x)$.
We define the \emph{co-degree} of $x$ at time $t$ as ${\deg^c}_t(x)=n_t-\deg_t(x)-1.$ It is straightforward to note that for $t\ge 1,$ $n_t = 2n_{t-1}=2^tn_0.$ Further, for an existing node $x\in
V(G_t)$, $\deg_{t+1}(x) = n_t-1$ and $\deg_{t+1}(x') = {\deg^c}_t(x).$

The ILAT model generates densification as we prove next. While the proof is elementary, the result is not a priori obvious from the model. One interpretation is that in networks where
anti-transitivity is pervasive, we expect that many alliances form in the network over time.
\begin{theorem}\label{denser}
The ratio $e_t/n_t$ tends to infinity with $t.$
\end{theorem}

\begin{proof} Note that by the definition of the model, we have that

\begin{eqnarray*}
e_{t+1} &=& e_t + \sum_{x\in V(G_t)} {\deg_t}^c(x) \\
& = & e_t + {n_t}^2 - 2e_t-n_t \\
& = & {n_t}^2 - e_t - n_t.
\end{eqnarray*}

Solving this recurrence, we derive that

\begin{eqnarray*}
e_{t} &=& {n_{t-1}}^2\left(\frac{4}{5}\right)\left(1-\left({-\frac{1}{4}}^{t-1}\right)\right)- n_{t-1}\left(\frac{2}{3}\right)\left(1-\left({-\frac{1}{2}}^{t-1}\right)\right) \\
& = & 2^{2t}\left(\frac{1}{5}\right)\left(1-\left({-\frac{1}{4}}^{t-1}\right)\right)(1-o(1)).
\end{eqnarray*}

Hence, we obtain that $e_t/n_t = \Omega (2^t)$.  \qed \end{proof}

\subsection{Degree distribution and density}

We next consider the degree distribution of the graph $G_t.$ For each node $x$ that at time $t$, we create its anti-clone $x'$ at time $t+1.$ Then at time $t+2$ we create $x''$ from $x$ and $(x')'$
from $x'.$ For any node $x$ that was created at a time-step $k<t$, we have that
$$\deg_t(x)=\frac{n_t}{2}-2.$$ To see this, notice that the graph at time step $t\geq 1$ can be partitioned into two halves: nodes $y$ that existed at time step $t-1$, and their newly created
clones $y'$. For each pair $y,y'$, we have that $x$ is adjacent to exactly one node in this pair. \bigskip

If $t>1$, then of the newly created nodes, half are anti-clones $x'$ of nodes $x$ that have already existed at time $t-2$, and therefore, their degree at time $t-1$ was
$$\deg_{t-1}(x)=\frac{n_{t-1}}{2}-2=\frac{n_t}{4}-2.$$ These anti-clones have at time $t$, $$\deg_t(x')=n_{t-1}-\deg_{t-1}(x)=\frac{n_t}{4}+2.$$ Similarly, if $t>2$ then there are $\frac{n_t}{8}$ nodes
$y''$ created at time $t$ that are anti-clones of nodes $y'$ created at time $t-1$ from nodes $y$ at least as old as $t-3$. Then since by the previous argument $\deg_{t-1}(y')=\frac{n_{t-1}}{4}+2,$
we have that $$\deg_t(y'')=\frac{3n_t}{8}-2.$$ If we continue in this fashion, then by induction we will find that at time $t$, we have that $2^{-k}n_t$ nodes of degree $a_k+(-1)^{k-1}2$ provided that for $k<t$:
$$a_1=\frac{n_t}{2}-2,$$
and
$$a_k=\frac{1}{2}-\frac{a_{k-1}}{2}.$$

From this discussion, we can obtain a limiting density as $t\rightarrow\infty$. Let $D_t$ be the density of $G_t;$ that is, $D_t = \frac{e_t}{\binom{n_t}{2}}.$ Parallel with Theorem~\ref{denser},
The ILAT model generates quite dense graphs.
\begin{theorem}\label{limit}
As $t\rightarrow \infty$, we have that $D_t \rightarrow 2/5.$
\end{theorem}

\begin{proof} Consider the sequence below, which describes the proportion of the whole graph represented by nodes of a certain
degree and the fraction of the nodes of the graph these particular nodes are adjacent to:
$$\frac{1}{2}\frac{1}{2}+\frac{1}{4}\frac{1}{4}+\frac{1}{8}\frac{3}{8}+\dots +\frac{1}{2^i}a_i+\dots.$$ Hence, $$\lim_{t\rightarrow \infty}D_t =  \sum_{i=1}^{\infty}\frac{1}{2^i}a_i=\sum_{i=1}^\infty b_i ,$$
where
$$b_i=\frac{1}{2^i}\frac{1-a_{i-1}}{2}=\frac{1}{2^{i-1}}\left(\frac{1}{4}-\frac{a_i}{4}\right)=\frac{1}{4}\left(\frac{1}{2^{i-1}}-b_{i-1}\right),$$ with $b_0=0$ and $b_1=\frac{1}{4}$.

We may now find the limiting density of the graph as $t\rightarrow\infty$:
$$\lim_{t\rightarrow \infty}D_t = \sum_{i=1}^\infty b_i=\lim_{k\rightarrow\infty} \sum_{i=1}^k b_i,$$ where $$ \sum_{i=1}^k b_i =\sum_{i=1}^k
\frac{1}{4}\left(\frac{1}{2^{i-1}}-b_{i-1}\right)=\frac{1}{4}\sum_{i=1}^k\frac{1}{2^{i-1}}-\frac{1}{4}\sum_{i=1}^kb_{i-1}. $$ Therefore, we have that
$$\sum_{i=1}^kb_i+\frac{1}{4}\sum_{i=1}^kb_{i-1}=\frac{1}{4}\sum_{i=1}^k\frac{1}{2^{i-1}}=\frac{1}{2}, $$ and so $\frac{5}{4}\sum_{i=1}^kb_i-b_k=\frac{1}{2}.$ As $b_k = o(1)$ we have that
 $$\lim_{k\rightarrow \infty}\sum_{i=1}^kb_i= \frac{2}{5},$$ and the proof follows.
 \qed \end{proof}

As an alternative way of obtaining the limiting density, suppose for the sake of the argument that a limiting density of $G_t$ as $t\rightarrow\infty$ exists. Then we have that
$$e_t=e_{t-1}+2\left({n_{t-1} \choose 2}-e_{t-1}\right),$$ since for every pair of non-adjacent nodes $x,y$ in $G_{t-1}$, two new edges are created: $xy'$ and $x'y$. Hence, if
 the limit exists: $$\frac{2{n_{t-1} \choose 2}-e_{t-1}}{{n_t\choose 2}}=\frac{e_{t-1}}{{n_{t-1}\choose 2}}, $$ and
$$\frac{2{n_{t-1} \choose 2}}{{n_t\choose 2}}=\frac{e_{t-1}}{{n_{t-1}\choose 2}}+\frac{e_{t-1}}{{n_t\choose 2}}, $$ for large $t$,
${n_t\choose 2}\sim 4{n_{t-1}\choose 2}$. Thus, we have that $\frac{1}{2}=\frac{5e_{t-1}}{4{n_{t-1}\choose 2}},$ and we find that $D_{t-1}=\frac{e_{t-1}}{{n_{t-1}\choose 2}}=\frac{2}{5}.$

\section{Distances and graph parameters}\label{distance}

The distances within graphs generated by ILAT become very small, with diameter $3$. Hence, highly anti-transitive networks exhibit short paths between nodes; this occurs at time-step $t=2$,
regardless of the starting diameter of $G_0.$

\begin{theorem} \label{diam}
Let $t\geq 2$, then the diameter $\mathrm{diam}(G_t)$ of $G_t$ is 3.
\end{theorem}
Note that the value $t=2$ in Theorem~\ref{diam} is sharp. For example, we may take $G_0$ to be a path of length 4. Or we may consider an initial graph of $K_3,$ in which case the graph at $t=1$ is
disconnected.

\bigskip

\noindent
\emph{Proof of Theorem~\ref{diam}.} We show first that for $t\geq 1$, the diameter of $G_t$ is at least 3. To see this,
consider the distance between some node $x$ that existed at time $t-1$ and its anti-clone $x'$ created at time $t$. They are not adjacent and have no common neighbors, and so we have that
$d(x,x')\geq 3$.

We next show that for $t\ge 2,$ any two nodes that are not newly created are at most distance 2 apart. For this, let $x, y$ be two distinct nodes that already existed at time $t-1$. Since the node
degree at time $t-1$ is bounded by $n/4-2$, by the pigeonhole principle there is another node $z$ that also existed at $t-1$ that is not adjacent to either of them. Hence, $z'$ is adjacent to both nodes
and so $d(x,y)\leq 2$.

Let $x',\ y'$ be two separate nodes newly anti-cloned from some nodes $x,\ y$. Since the node degree at time $t-1$ is bounded by $\max\{ 0,n/4-2 \}$, by the pigeonhole principle there is another node $z$ that
also existed at $t-1$ that is not adjacent to either $x$ or $y$. Then $z$ is adjacent to both $x'$ and $y'$, and so $d(x',y')\leq 2$. Hence, any two nodes that both newly created are at most distance 2
apart.

The only case we have not considered are pairs of nodes where one is newly created and one is not. But if $t\geq 3$, then every newly created node has a neighbor that is not newly created and vice
versa. Therefore, any such pair can be connected by a path of length at most 3. \qed

\bigskip

For the average distance in ILAT graphs, we would like to show that all but a negligible number of pairs of nodes have distance at most 2. Let $L_t$ denote the average distance at time $t.$ Since we
know the limiting density is 0.4 by Theorem~\ref{limit}, we have for a constant d
$$\lim_{t\rightarrow\infty}L_t=\lim_{t\rightarrow\infty}\left(0.4\cdot 1+\frac{d}{2^t}\cdot 3+\left(1-0.4-\frac{d}{2^t}\right)\cdot 2\right)=1.6.$$

The pairs of nodes we have not considered so far are ones where exactly one node is newly created, but is not a anti-clone of the other. If they are not adjacent, then we would like to know if they
have a common neighbor. Let the node that already existed at time $t-1$ be $x$, and the newly created node be $y'$, cloned from some node $y\not=x$. Nodes $x$ and $y'$ can have a common neighbor unless the
neighborhood of $x$ at time $t-1$ (other than possibly $y$ itself) was a subset of the neighborhood of $y$ at time $t-1$ (which would be the case when $x=y$).

\begin{theorem}
If $x$ and $y$ are nodes of $G_t$ that are not newly created at time $t$, with $t\geq 2$ and $x\not=y$, and it is not the case that both $x$ and $y$ belonged to $G_0$, then $d(x,y')\leq2$.
\end{theorem}

\begin{proof} Unless $x$ and $y$ are adjacent, we have that $d(x,y')=1$. So suppose that $x$ and $y$ are adjacent. Suppose that they did not both belong to the initial graph $G_0$. Since they are adjacent, one of them was
created later than the other. If $y$ was created later, then every neighbor of $x$ that was created at the same time as $y$ is now a common neighbor of $x$ and $y'$. If $x$ was created later, but
before $t-1$, then every node adjacent to $y$ but not $x$ at the time produced a anti-clone of the type we need. We are left with a case where $x$ was created at time $t-1$, and $y$ was created
earlier.

We want to find a common neighbor of $x$ and $y'$ that was created at $t-2$ or earlier. $x$ was created at time $t-1$, so it was cloned from a node with has either $n/8-2,\ n/16+2$ or about $n/12$
neighbors that already existed at time $t-1$, and so $x$ has either $n/8+2,\ 3n/16-2,$ or about $n/6$ neighbors older than itself. By the same argument, $y'$ has either $n/8+2,\ 3n/16-2,$ or about
$n/6$ neighbors at least as old as $t-2$. There are in total $n/4$ nodes at least as old as $t-2.$ So by the pigeonhole principle, they must have such a neighbor in common. \qed \end{proof}

Notice that the number of pairs such that both $x$ and $y$ belong to $G_0$ is negligible, so will not change the average distance limit. We can conclude that: $$\lim_{k\rightarrow\infty}L_t= 1.6.$$

We next turn to a brief discussion of the domination and cop numbers of the ILAT graphs. As we have noticed with other parameters such a the diameter and average distance, these two parameters are
bounded above by very small constants. For more on these graph parameters, see \cite{bonato} (we omit their definitions here as they are well-known and owing to space constraints). As a possible
interpretation of these, we note that in networks exhibiting high anti-transitivity, a few important nodes emerge (either dominating nodes, or mobile agents represented by cops) which can reach all
other nodes. Such so-called \emph{superpower} nodes organically emerge as important actors in the network.

\begin{theorem}
In $G_t$ such that $t\geq 3$, the domination number is $3$.
\end{theorem}

\begin{proof}Let $A=\{x, x', (x')'\}$ be as follows. For any $1\leq k\leq t-1$, let $x$ be a node that existed at time $k-1$ and $x'$ be the time-$k$ anti-clone of $x$. Let $x''$ be the
time-$(k+1)$ anti-clone of $x'$. Then any node of $G_t$ not in $A$ is either adjacent to $x'$, adjacent to $x''$, or a node created at time $k+1$ that is not adjacent to $x'$, in which case it must be
adjacent to $x$. Therefore, $A$ is a dominating set of $G_t$.

If $t\geq 1$, then we can never find a dominating set of size 2. The node degrees are bounded by $\frac{n_t}{2}-2$. Therefore the union of neighborhoods of any two nodes contains at most $n_t-4$
nodes. \qed \end{proof}

\begin{theorem}\label{copc}
If $t\geq 2$, then the cop number of $G_t$ is 2.
\end{theorem}

\begin{proof} In a simple, omitted argument, if $t\ge 2,$ the cop number of $G_t$ is never 1. We now describe how two cops can may capture the robber. Fix $v\in V(G_{t-1})$. Then each vertex of $G_{t-1}$ is adjacent to one of $v$ or $v'.$ Place the cops on $v$ and $v'$. Hence, the robber must begin on an anti-clone say $u'$ newly created at time $t$ not adjacent to either $v$ or $v'$. Now there must be an $x$ in $G_t$ joined to $u'$, otherwise, $u$ is a universal vertex in $G_{t-1}$ which is a contradiction (here is where we use $t\ge 2).$  It is straightforward to show that there is a perfect matching between $x,\ x'$ and $v, \ v'$, and so the cops move to $x, \ x'$. The robber must move to a vertex $z$ in $G_{t-1}$. But $z$ is joined to one of $x$ or $x'$ and the robber is caught in the next move. \qed \end{proof}

\bigskip
Note that we must have $t\ge 2$ in Theorem~\ref{copc} or the cop number could be larger than 2. For example, if $G_0$ is a $K_3,$ then $G_1$ is the disjoint union of $K_3$ and $\overline{K_3}$, which has cop number 4.


\section{Clustering coefficient}\label{cluster}

For a node $v$, define $c_t(v)$ to be the (local) clustering coefficient of the node $v$ at time $t.$ We note that in the ILAT model, older nodes exhibit significant local clustering over time.

\begin{theorem}\label{ccc}
Let $k\in \mathbb{N}$. For node $v$ created at time $k$, with $t> k$, if $\lim_{t\rightarrow\infty}c_t(v)$ exists, then we have that
$$\lim_{t\rightarrow\infty} c_t(v)=0.4.$$
\end{theorem}
Hence, the clustering coefficient of a node $v$ tends to $0.4$ as $v$ grows old, which matches the density of the graph.
\bigskip

\noindent \emph{Proof of Theorem~\ref{ccc}.} Let $c_t'(v)=c_t'$ be the density of $v$'s non-neighbor-hood set at time $t,$ and let $c_t''(v)=c_t''$ be the density between the neighborhood and the non-neighborhood of $v$. Hence, if $\deg_t(v)$ is the degree of $v$
at time $t$ and $n$ is the number of vertices at time $t$, then the number of edges with both endpoints in the neighborhood of $v$ is $c_t(v){{\deg_t(v)} \choose 2}$, the number of edges with both endpoints in the non-neighborhood of $v$ is
$c'_t{{n-\deg(v)-1} \choose 2}$, the number of edges with one endpoint in the neighborhood of $v$, and the remaining number of edges in the non-neighborhood of $v$ is $c''_t\deg(v)(n-\deg(v)-1)$.

For large $t$, we may approximate the degree by $\deg_t(v)\sim n-\deg(v)-1\sim \frac{n}{2}.$ Further, since the total number of edges in the graph tends to $0.4{n\choose 2}$, we have that $$\frac{c_t+c't+2c''_t}{4}\sim
\frac{2}{5},$$ and
$$c_t'\sim \frac{8}{5}-c_t-2c_t'' .$$ Then we may determine $c_{t+1}(v)=c_{t+1}$ by counting the edges with both endpoints in the neighborhood of $v$ at time $t+1$. These are either the same edges that
contributed to $c_t(v)$, or edges between the $t$-time neighborhood of $v$ and the anti-clones of its non-neighborhood, giving the following equations:
\begin{eqnarray*}
c_{t+1}{n\choose 2}&\sim &c_t{{n/2}\choose 2}+(1-c_t'')\frac{n^2}{4}, \\
c_{t+1}&\sim &\frac{c_t}{4}+\frac{1-c_t''}{2}.
\end{eqnarray*}
Further, we have that
\begin{eqnarray*}
c_{t+1}''&=&\frac{c_t''}{4}+\frac{1-c_t'}{4}+\frac{1-c_t}{4} \\
c_{t+1}''&=&\frac{c_t''}{4}+\frac{1-\frac{2}{5}+c_t(v)+2c_t''}{4}+\frac{1-c_t}{4}, \text{ and}\\
c_{t+1}''&=&\frac{3c_t''+\frac{2}{5}}{4} .
\end{eqnarray*}
By hypothesis, the limiting value of $c_t$ exists and we call this quantity $c$. In particular, we have that for a sufficiently large $t$ that,
$c_t(v)\sim c_{t+1}\sim c_{t+1} \sim c.$
We have that
$$c_{t+2}=\frac{c_{t+1}}{4}+\frac{1-c_{t+1}''}{2}= \frac{c_{t+1}}{4}+\frac{3}{4}\frac{1-c_t''}{2}+\frac{1-\frac{2}{5}}{8},$$
and so $c_{t+2}=c_{t+1}-\frac{3c_t}{16}+\frac{3}{40}.$ By taking the limit as $t\rightarrow \infty$, we have that $\frac{3}{16}c = \frac{3}{40},$ and the result follows.
\qed
\bigskip

We present bounds on the (global) clustering coefficient of $G_t$, denoted $C_t.$ Note that the clustering coefficients here are less than 0.4 (that is, the limiting graph density) and so the ILAT graphs have lower clustering coefficients than in binomial random graphs with the same average degree (unlike in the ILT model; see \cite{ilt}). The proof of this result is omitted and will be presented in the full version of the paper.

\begin{theorem}\label{tcluster}
For $t$ sufficiently large, we have that $$0.1100<C_t<0.1244. $$
\end{theorem}

\section{Spectral expansion}

For a graph $G=(V,E)$ and sets of nodes $X,Y \subseteq V$, define $E(X,Y)$ to be the set of edges in $G$ with one endpoint in $X$ and the other in $Y.$ For simplicity, we write $E(X)=E(X,X).$ The normalized Laplacian of a graph relates to important graph properties; see \cite{sgt} for a reference. Let $A$ denote the adjacency matrix and $D$ denote the
diagonal degree matrix of a graph $G$. Then the normalized Laplacian of $G$ is $\mathcal{L} = I - D^{-1/2}AD^{-1/2}.$ Let $0 = \lambda_0 \leq \lambda_1 \leq \cdots \leq \lambda_{n-1} \leq 2$ denote
the eigenvalues of $\mathcal{L}$. The \emph{spectral gap} of the normalized Laplacian is defined as
$$
\lambda = \max\{ |\lambda_1 - 1|, |\lambda_{n-1} - 1| \}.
$$
A spectral gap bounded away from zero is an indication of bad expansion properties, which is characteristic for social networks; see \cite{estrada}. The next theorem represents a drastic departure
from the good expansion found in binomial random graphs, where $\lambda = o(1)$; see~\cite{sgt,CL}.
\begin{theorem}\label{thm:expa}
$\lambda \ge 3/5 + o(1).$
\end{theorem}

To prove Theorem~\ref{thm:expa}, we use the expander mixing lemma for the normalized Laplacian (see~\cite{sgt} for its proof). For sets of nodes $X$ and $Y$ we use the notation $\vol(X) = \sum_{v \in
X} \deg(v)$ for the volume of $X$, $\bar{X} = V \setminus X$ for the complement of $X$, and, $e(X,Y)$ for the number of edges with one end in each of $X$ and $Y.$ (Note that $X \cap Y$ does not have
to be empty; in general, $e(X,Y)$ is defined to be the number of edges between $X \setminus Y$ to $Y$ plus twice the number of edges that contain only nodes of $X \cap Y$. In particular, $e(X,X) = 2
|E(X)|$.)
\begin{lemma}\label{mix}
For all sets $X \subseteq G,$
\[
\left| e(X,X) - \frac{(\vol(X))^{2}}{\vol(G)} \right| \leq \lambda \frac{\vol(X)\vol(\bar{X})}{\vol(G)}.
\]
\end{lemma}

\bigskip

\noindent \emph{Proof of Theorem~\ref{thm:expa}.} Let $X$ be the set of $n/2$ the youngest nodes. Since $X$ induces an independent set, we note that $e(X,X)=0$. We derive that
\begin{eqnarray*}
\vol(G) &\sim& 2n^2 / 5, \\
\vol(\bar{X}) &\sim& n^2 / 4, \quad \text{and}\\
\vol(X) &=& \vol(G) - \vol(\bar{X}) \sim 3 n^2 / 20,
\end{eqnarray*}
where the second expression holds as $(n/2)$-many of the oldest nodes have degree $\sim n/2$. Hence, by Lemma~\ref{mix}, we have that
$$
\lambda \ge \frac{(\vol(X))^{2}}{\vol(G)} \cdot \frac{\vol(G)}{\vol(X)\vol(\bar{X})} =  \frac{\vol(X)}{\vol(\bar{X})} \sim 3/5,
$$
and the proof follows. \qed

\section{Discussion and future work}\label{conc}

We introduced the Iterated Local Anti-Transitivity (ILAT) model for complex networks and analyzed properties of the graphs it generates. We proved that graphs generated by ILAT densify over time,
have diameter 3, and have density tending to 0.4. ILAT graphs have small dominating sets and low cop number. We analyzed the clustering coefficient of ILAT graphs, and noted that while older nodes show high (local) clustering, the (global) clustering coefficient is less than what is expected in binomial random graphs with the same expected degree. In addition, we showed that graphs generated by ILAT exhibit bad spectral
expansion as found in social networks.

Theoretical results presented here for the ILAT model are suggestive of several emergent properties in networks where anti-transitivity governs link formation. For instance, the presence of small
(3-element) dominating sets suggest the emergence of nodes we describe as \emph{superpowers}, which have broad influence in the network. Such nodes may emerge naturally in real-world networks which are highly
anti-transitive, owing to a high number of alliances against common adversaries. Similarly, the presence of short paths, high density, and high (local) clustering of older nodes in ILAT graphs suggests that networks, where
common adversaries forge alliances, naturally form tight-knit communities that are well-connected. In the sequel, it would be interesting to empirically test these hypotheses with real-world networked
data.

Besides applications of the ILAT model, it raises a number of interesting graph-theoretic questions. An open problem remains to compute the exact clustering coefficient for ILAT graphs. Another
question is to determine the induced subgraph structure of such graphs. A characterization of the induced subgraphs of ILAT graphs (that is, to determine its \emph{age}) remains open. For example, do
all finite trees appear as induced subgraphs of ILAT graphs?

\end{document}